\documentclass{iopart}
\usepackage{iopams}
\usepackage{geometry}                
\usepackage{graphicx}
\usepackage{amsthm}

\expandafter\let\csname equation*\endcsname\relax
\expandafter\let\csname endequation*\endcsname\relax
\usepackage{amsmath}

\usepackage{epstopdf}
\usepackage{hyperref}
\usepackage{cite}
\usepackage[percent]{overpic}

\newtheorem{proposition}{Proposition}
\newtheorem{theorem}{Theorem}

\def\R{\mathbb{R}}
\def\D{\mathrm{D}}

\begin{document}



\title{Birational maps from polarization and the preservation of measure and integrals}
\author{Robert I McLachlan$^1$\footnote{Author to whom any correspondence should be addressed.}, David I McLaren$^2$,  and G R W Quispel$^2$}
\address{%
    $^1$School of Mathematical and Computational Sciences, Massey University, Palmerston North 4472, New Zealand\\%
    $^2$Department of Mathematical and Physical Sciences, La Trobe University, Melbourne, Victoria 3086, Australia\\%
    }

\begin{abstract}
\noindent
The main result of this paper is the discretization of second-order Hamiltonian systems of the form
$\ddot x = -K \nabla W(x)$, where $K$ is a constant symmetric matrix and $W\colon\R^n\to \R$ is a polynomial of
degree $d\le 4$ in any number of variables $n$. The discretization uses the method of polarization and preserves both the energy and the invariant measure of the differential equation, as well as the dimension of the phase space. This generalises earlier work for discretizations of first order systems with $d=3$, and of second order systems with $d=4$ and $n=1$.
\end{abstract}

\section{Introduction}
\subsection{Preamble}

To any physicist, invariance and covariance are important properties. The significance of symmetries, integrals, symplectic structure, measure preservation, etc., can hardly be overstated. For example, the above all play a role in Hamilton's equations. A fact of life is that, in most instances, Hamilton's equations cannot be solved in closed form, and one must rely on numerical computation.

A problem then arises. Most traditional numerical integration methods do not preserve symmetries, integrals, symplectic structure, etc. This has led to the development of a new area of computational physics, entitled geometric numerical integration, in which methods are developed to preserve each of the aforementioned properties (and more). For surveys of this area see \cite{mcqu,qumc}.

What has turned out to be more difficult is to preserve several properties at the same time. In particular, preserving the Hamiltonian and phase space volume simultaneously has proved to be hard in general. This situation improved when it was discovered that Kahan's method does preserve an integral as well as a volume form (albeit that it exactly preserves a nearby integral and a nearby measure) \cite{ka2,sa,ce6,pe9a}. This has led to a flurry of research extending and generalizing these results
\cite{ce3,ce5,ce6,ce7,pe7,pe9a,pe9,pe14}. The present paper follows in this tradition.

\subsection{Darboux polynomials and birational integrators}
The Cremona group of birational maps of a vector space---rational maps with a rational inverse---is studied both  for its properties as an infinite  group (for example, its generators and subgroups) and in terms of the dynamics of its elements \cite{ca,cr,fa}. (The H\'enon map $(x,y)\mapsto (y+x^2+c, ax)$ is an example.)
Like the Cremona group, Darboux polynomials were also introduced in the nineteenth century, as a tool for the analysis of ordinary differential equations \cite{da}. Yet only in recent years has it become apparent that there is a connection between these two objects. 
Gasull and Ma\~nosa \cite{ga} constructed first integrals for some planar maps using Darboux polynomials. Celledoni et al.  \cite{ce2,ce1} gave an effective method for the construction of Darboux polynomials of birational maps, which in turn gave a method for the detection and determination of first integrals and invariant measures of birational maps.

In a parallel development, interesting new examples of birational maps have been discovered based on specific discretizations of ordinary differential equations \cite{
ce3,ce5,ce6,ce7,es,gu1,gu2,hi1,ho2,ia,ka2,ka1,ki,ko,pe7,pe9,pe9a,pe14,qu1,sa,va,wa,za3}. The Kahan (or Kahan--Hirota--Kimura) \cite{ka2,ka1} map is an example.
In some cases these maps have been found to preserve first integrals, invariant measures, and/or integrability, which is a decidedly exceptional feature of a numerical discretization. Typically, the preserved structures are perturbations in the time step of the original ones.

An example is the Kahan map applied to a canonical Hamiltonian system with a cubic Hamiltonian \cite{ce6}. Regardless of the dimension, the Kahan map has an invariant measure and first integral, which are perturbations of those of the continuous system. So in dimension 2, it generates integrable maps; in dimension 4 and greater, it  provides examples of
maps with nonlinear integrals and conserved measures unrelated (in general) to integrability
or to obvious symmetries,  a novel feature.

Darboux polynomials provide a route to discover new classes of such maps. In this paper we study maps
associated with Hamiltonian systems of the form
\begin{eqnarray}
\label{eq:ode}
 \ddot x = -K \nabla W(x),\quad x\in\mathbb{R}^{n},
 \end{eqnarray}
 where $K$ is a constant symmetric matrix and $W\colon\R^{n}\to\R$ is a polynomial of degree $\le 4$. These are canonical Hamiltonian systems with Hamiltonian
\begin{eqnarray}
\label{eq:ham} H(x,p)=\frac{1}{2}p^\top K p + W(x),
\end{eqnarray}
 written in second-order form. The maps themselves take the form
 \begin{eqnarray}
 \label{eq:m1} \frac{x_{k+1}-2x_k+x_{k-1}}{h^2} = f(x_{k-1},x_k,x_{k+1})
 \end{eqnarray}
 where $f$ is linear in each argument, thus forming a family of birational maps parametrised by $h$ with the same phase space as the associated continuous system. The function $f$ is constructed from $W$ by ``polarization'', a standard technique in algebraic geometry that has been used to construct birational mappings from differential equations \cite{ce5} and which provides a new viewpoint on the Kahan mapping. Its use  in the form (\ref{eq:m1}) was introduced by Hone and Quispel \cite{ho2}.  
 
 Our method is the following:
 \begin{enumerate}
 \item Find Darboux polynomials for specific instances or specific parametric families of maps (\ref{eq:m1}).
 \item For cases such that the Darboux polynomials provide invariant integrals or measures of the maps, conjecture the general form of these integrals and measures.
 \item Prove the invariance of these integrals and measures directly.
 \end{enumerate}
 
 Sections 2 and 3 review Darboux polynomials, the Kahan mapping, and polarization. In Section 4 we establish the existence of an invariant measure for the polar map on $\mathbb{R}^{2n}$ associated with (\ref{eq:ode}) when the potential $W(x)$ is a homogeneous quartic in $n$ variables, while Section 5 gives a first integral for this case.
 In Section 6 these results are extended to the nonhomogeneous case and (for the invariant measure only) to non-conservative vector fields. Section 7 extends the results to polar maps associated with higher-order systems
 $x^{(m)}=f(x)$, generalizing in addition a  known result for the Kahan mapping in the case $m=1$. Two results in Section 8 point towards possible further generalizations: a polar map with a two  integrals, and a polar map associated with a 4-parameter family of 4th order systems that has an integral.

It is not known which birational maps have invariants like the ones we find here, or how the specific functional form of the integral is related to the construction of the map. (We have no analogue of Noether's theorem, for example.) We hope that our results here will lead in the direction of greater understanding of these questions.
 
 \section{Darboux polynomials for differential equations and maps}

A {\em continuous Darboux polynomial} for an ODE  $\dot x = f(x)$, $x\in\R^n$, is a polynomial $P\colon\R^n\to \R$ such that there exists a polynomial $C\colon\R^n\to\R$, called the {\em cofactor of $P$}, such that $\dot P = C P$ where $\dot P := (\nabla P)^\top f$. The level set $P^{-1}(0)$ is
then invariant. Darboux polynomials have also been called `second integrals' and `weak integrals' of $f$. If $P_1$ and $P_2$ are Darboux polynomials with the same cofactor $C$, then $P_1/P_2$ is a first integral of $f$. 

In addition, if $\dot P_i = C_i P_i$, then $\frac{d}{dt}(P_1P_2) = (C_1+C_2)P_1P_2$, i.e., the product $P_1P_2$ of two Darboux polynomials is a Darboux polynomial with cofactor equal to the {\em sum} $C_1+C_2$. Moreover, if $P$ is a Darboux polynomial for the polynomial vector field $f$, then all irreducible factors of $P$ are Darboux polynomials \cite{go}.

However, even if $P_1$, $P_2$, and $C$ are known to be polynomials of a given degree, determining their coefficients such that $\dot P_1 = CP_1$ and $\dot P_2 = C P_2$ is a nonlinear problem. This has tended to limit the discovery of  Darboux polynomials to cases where either $P$ or $C$ is particularly simple (e.g. linear), or where an invariant set which may be identified as $P^{-1}(0)$ can be found by other means.

A {\em discrete  Darboux polynomial} for a map $x\mapsto x' :=\phi(x)$, $x\in\R^n$, is a polynomial $P\colon\R^n\to \R$ such that there exists a rational function $C\colon\R^n\to\R$ (again called the cofactor of $P$) such that $P' = C P$ where $P':=P\circ\phi$. As in the continuous case, the level set $P^{-1}(0)$ is then invariant, and if $P_1$ and $P_2$ are Darboux polynomials with the same cofactor $C$, then $P_1/P_2$ is a first integral of $\phi$. 

The key difference between the discrete and the continuous cases is that in the discrete case if $P_i'=C_i P_i$ then $P_1'P_2' = (C_1 C_2) (P_1P_2)$, i.e., the product $P_1P_2$ of two Darboux polynomials is a Darboux polynomial with cofactor equal to the {\em product} $C_1C_2$.  

The maps $\phi$ in which we are interested are birational and not polynomial \cite{pa,vi}. In this case, an excellent ansatz is available for cofactors of birational maps, namely, that their factors are factors of the Jacobian determinant of the map. Therefore, the unique factorization of the Jacobian determinant of the map yields a finite number of choices of cofactor up to some chosen degree, which is not possible in the continuous case. For each such choice, the determination of the Darboux polynomials associated with that cofactor is a linear problem.

In addition, we have:
\begin{enumerate}
\item If $P$ is a Darboux polynomial with cofactor $1/\det \D\phi$, then $\frac{1}{P}dx_1\wedge\dots\wedge dx_n$ is an invariant measure of $\phi$.
\item If $P_1,\dots,P_k$ are Darboux polynomials with cofactors $C_1,\dots,C_k$, respectively, 
and $\prod_{i=1}^k C_i^{\alpha_i}=\det \D\phi$, then $\big(\prod_{i=1}^k P_i^{\alpha_i}\big)dx_1\wedge\dots\wedge dx_n$ is an invariant measure of $\phi$.
\item If $P_1,\dots,P_k$ are Darboux polynomials with cofactors $C_1,\dots, C_k$, respectively, and
$\prod_{i=1}^k C_i^{\alpha_i}=1$, then $\prod_{i=1}^k P_i^{\alpha_i}$ is a first integral of $\phi$.
\item If $P_1,\dots,P_k$ are Darboux polynomials with the same cofactor $C$, then $a_1 P_1+\dots+a_k P_k$ is a Darboux polynomial with cofactor $C$ for any constants $a_1,\dots,a_k$. That is, the Darboux polynomials for a given cofactor form a vector space.
\end{enumerate}

\section{Kahan's method and polarization}
Kahan's method \cite{ka2,ka1} for the quadratic ODE 
\begin{eqnarray}
\dot x = f(x) := Q(x,x)+Bx+c,\quad x\in\R^n
\label{eq:quadvf}
\end{eqnarray}
(where $Q$ is an $\R^n$-valued symmetric bilinear  form, $B\in\R^{n\times n}$, and $c\in\R^n$) is
the map  $x\mapsto x'$ with step size $h$ given by
\begin{eqnarray}
\frac{x'-x}{h} = Q(x,x') + \frac{1}{2}B(x+x')+ c.
\label{eq:kahanmap}
\end{eqnarray}
Because the right-hand side of Eq. (\ref{eq:kahanmap}) is linear in $x'$, it can be solved explicitly to get the rational map
\begin{eqnarray*}
x'= x+h\Big(I - \frac{h}{2}f'(x)\Big)^{-1}f(x).
\end{eqnarray*}
Its inverse is also rational,
\begin{eqnarray*}
x= x'-h\Big(I + \frac{h}{2}f'(x')\Big)^{-1}f(x'),
\end{eqnarray*}
so the Kahan map is birational. 

Moreover, a homogeneous quadratic vector field $f(x)$ can be expressed in terms of a
bilinear form $Q(x, x)$, as in Eq. (\ref{eq:quadvf}), using the technique of polarization:
\begin{eqnarray*} Q(x_1,x_2) = \frac{1}{2}(f(x_1+x_2)-f(x_1)-f(x_2)).\end{eqnarray*}
So the Kahan method can be obtained by polarizing the quadratic terms of the ODE, evaluating
them at $(x,x')$ and by replacing the linear and constant terms by the midpoint approximation.

Polarization is a map from a homogeneous polynomial to a symmetric multilinear form in
more variables. For example, the polarization of the cubic $f(x)$ is the trilinear form
\begin{eqnarray*}F(x_1, x_2, x_3) = \frac{1}{6}\frac{\partial}{\partial\lambda_1}\frac{\partial}{\partial\lambda_2}\frac{\partial}{\partial\lambda_3}\left.f(\lambda_1 x_1 + \lambda_2 x_2 + \lambda_3 x_3)\right|_{\lambda=0},\end{eqnarray*}
where $x$, $x_1$, $x_2$, $x_3\in\mathbb{R}^n$. This is equal to $\frac{1}{6}$ 
 of the coefficient of $\lambda_1\lambda_2\lambda_3$  in
$f(\lambda_1 x_1 + \lambda_2 x_2 + \lambda_3 x_3).$ It satisfies
$F(x, x, x) = f(x)$.

For example, consider $x=(y,z,w)\in\mathbb{R}^3$.
The polarization of $3y^2z$ is
$y_1y_2z_3 + y_2y_3z_1 + y_3y_1z_2$ and the polarization of 
$6yzw$ is $y_1z_2w_3 + y_2z_3w_1 + y_3z_1w_2 + y_1z_3w_2 +
y_3z_2w_1 + y_2z_1w_3$. 

Polarization of a homogeneous vector field of degree $k + 1$ determines  a multilinear form in $k +
1$ variables. We call these variables $x_0,\dots,x_k$, where $x_k\in\mathbb{R}^n$.

Celledoni et al.  \cite{ce5} considered the discretization of the first-order homogeneous ODE
\begin{eqnarray*}\dot x = F(x, x, \dots , x), \quad x\in\mathbb{R}^n\end{eqnarray*}
of degree $k + 1$ by the birational `polar map' 
$(x_0, \dots , x_{k-1}) \to (x_1, \dots , x_k)$, where
\begin{eqnarray*} \frac{x_k-x_0}{h} = F(x_0, \dots , x_k),\end{eqnarray*}
regarded as a $k$-step method. When the ODE is canonical Hamiltonian (or Poisson with
a constant Poisson structure), 
this polar map has an invariant measure and its $k$th iterate has $k$ independent integrals.
For $k=1$ it coincides with Kahan's method, but when $k>1$ its phase space $\mathbb{R}^{nk}$ is different
from that of the ODE.

\section{Polarization of homogeneous second order ODEs: invariant measure}

In this paper, we consider the canonical Hamiltonian systems with phase space $\mathbb{R}^{2n}$, Hamiltonian (\ref{eq:ham}), and Hamilton's equations
\begin{eqnarray*}
\dot x &= K p,\\
\dot p &= -\nabla W(x).
\end{eqnarray*}
Eliminating $p$ gives the second-order ODE (\ref{eq:ode}). When $W(x)$ is a homogeneous polynomial of degree 4,
let $\frac{1}{4}V(x_0,x_1,x_2,x_3)$ be the polarization of $W$ so that $W(x)=\frac{1}{4}V(x,x,x,x)$ and $V$ is linear in each argument. Note that we then have
\begin{eqnarray}
\label{eq:hess}
\eqalign{
\nabla W(x) = V(x,x,x,\cdot),\\
\mathop{\rm Hess} W(x) = 3 V(x,x,\cdot,\cdot).
}
\end{eqnarray}
Here Hess$W(x)$ is the Hessian of $W$, i.e. the matrix of second partial derivatives of $W$.

We consider the second-order map $(x_0,x_1)\mapsto (x_1,x_2)$ defined by
\begin{eqnarray}
\label{eq:map}
 \frac{x_2-2x_1+x_0}{h^2} = - K V(x_0,x_1,x_2,\cdot).
\end{eqnarray}

It is known from the work of Hone and Quispel \cite{ho2} that the map (\ref{eq:map}) has an integral and a preserved measure for $n=1$. We now outline the strategy that we have used to discover  analogous invariants in any dimension $n$. 

We first consider specific examples of (\ref{eq:map}) for $n=2$. In these examples, we find that the Jacobian determinant factors as $N/D^3$ where $N$ and $D$ are polynomials. The cofactor ansatz described above then leads to the determination of two linearly independent Darboux polynomials, $P_4$ of degree 4 and $P_6$ of degree 6, with cofactors both equal to the Jacobian determinant. As they have the same cofactor, any linear combination of $P_4$ and $P_6$ is again a Darboux polynomial. Thus, we look for a convenient basis. As 
the degree 4 polynomial has nonzero constant term, the constant term in $P_6$ can be eliminated. 
In this basis, the terms in $P_6$ of degree 2, 4, and 6 can be expressed as invariant functions of  $K$ and $V$. Using this approach, it was possible to discover general expressions for two independent Darboux polynomials for any $n$, and hence our two key results: the invariant measure of the map (Proposition \ref{prop:measure}) and a first integral (Theorem \ref{thm:integral}).
 
\begin{proposition}
\label{prop:measure}
The map (\ref{eq:map}) on $\mathbb{R}^{2n}$ has invariant measure $(\det M_{01})^{-1}\prod dx$, where $\prod dx=dx_{0,1}\wedge\dots \wedge dx_{0,n}\wedge dx_{1,1}\wedge\dots\wedge dx_{1,n}$ is the
standard Euclidean measure on $\mathbb{R}^{2n}$ and
\begin{eqnarray*}M_{01} := I + h^2 K V_{01},\quad V_{01} := V(x_0,x_1,\cdot,\cdot).\end{eqnarray*}
\end{proposition}
\begin{proof}
Because $V(x_0,x_1,x_2,\cdot) = V_{01}x_2$, the  equation (\ref{eq:map}) for $x_2$ can be written in the form
\begin{eqnarray}
\label{eq:Mform}
 M_{01} x_2 = 2 x_1 - x_0.
 \end{eqnarray}
The map in first order form
\begin{eqnarray*} \begin{pmatrix}x_1\\  x_0 \end{pmatrix}
\mapsto
\begin{pmatrix}x_2\\  x_1 \end{pmatrix}
\end{eqnarray*}
has Jacobian derivative
\begin{eqnarray*}\begin{pmatrix} \D_{x_1} x_2 & \D_{x_0} x_2 \\  I & 0  \end{pmatrix}\end{eqnarray*}
with determinant
\begin{eqnarray}
\label{eq:det}
\det(-\D_{x_0} x_2).
\end{eqnarray}
Taking Jacobian derivatives of both sides of (\ref{eq:Mform}) with respect to $x_0$ and using the product rule on the left gives
\begin{eqnarray*} M_{01} \D_{x_0} x_2+ h^2 K V_{12} = -I\end{eqnarray*}
or
\begin{eqnarray*} I+h^2 K V_{12} = M_{12} = -M_{01}D_{x_0} x_2 .\end{eqnarray*}
Taking determinants gives
\begin{eqnarray*}\det M_{12} = \det (-\D_{x_0} x_2)\det M_{01}\end{eqnarray*}
which, together with (\ref{eq:det}), establishes the result.
\end{proof}

\section{Polarization of homogeneous second order ODEs: first integral}

The following proposition establishes that the map (\ref{eq:map}) has a first integral
for all nonsingular matrices $K$ and for all homogeneous quartic potentials $V$, in any dimension. 
In the  expression (\ref{eq:int}) for the integral, under the approximation $x_1-x_0\approx h \dot x$,
the first term approximates $h^2\dot x^\top K^{-1}\dot x$, i.e. $2h^2$ times the kinetic energy, 
and the second term approximates $2h^2 W(x)$, i.e. $2h^2$ times the potential energy.
That is, the integral is a perturbation of the Hamiltonian of (\ref{eq:ham}).

\begin{theorem}
\label{thm:integral}
Let $K$ be a symmetric nonsingular $n\times n$ matrix. The map (\ref{eq:map}) has the rational first integral
\begin{eqnarray}
\label{eq:int}
F(x_0,x_1) := \Delta x_0^\top  K^{-1}M_{01}^{-1} \Delta x_0 - \frac{1}{2}x_0^\top K^{-1}(M_{01}^{-1}-I)x_1
\end{eqnarray}
where $\Delta$ is the forward difference operator, i.e. $\Delta x_0 = x_1-x_0$.
\end{theorem}
\begin{proof}
We need to show that $F(x_1,x_2)=F(x_0,x_1)$, that is, that $\Delta F(x_0,x_1)=0$.
We recall the product formula for forward differences,
\begin{eqnarray*}\Delta(a_0 b_0 c_0) = a_0 b_0 (\Delta c_0)+a_0 (\Delta b_0)c_1 + (\Delta a_0) b_1 c_1.\end{eqnarray*}
From Eq. (\ref{eq:map}), we have
\begin{eqnarray*} 
\Delta^2 x_0 &= - K V_{01} x_2 \\
&= (I-M_{01})x_2 \\
&= (I-M_{12})x_0,
\end{eqnarray*}
\begin{eqnarray*} 
M_{01}^{-1}(2x_1-x_0) &= x_2,
\end{eqnarray*}
and
\begin{eqnarray*} 
M_{12}^{-1}(2x_1-x_2)& = x_0.
\end{eqnarray*}
Note that both $M_{01}K$ (which is equal to $(I+h^2 KV_{01})K$) and its inverse $K^{-1}M_{01}^{-1}$ are symmetric matrices.

The forward difference of the first term in the integral $F$ is
\begin{eqnarray*}
\fl \Delta(\Delta x_0^\top M_{01}^{-1} \Delta x_0)
& = 
\Delta x_0^\top  K^{-1}M_{01}^{-1} \Delta^2 x_0 + \Delta x_0^\top K^{-1} \Delta M_{01}^{-1} \Delta x_1 + \Delta^2 x_0^\top K^{-1}M_{12}^{-1} \Delta x_1\\
&= (x_1-x_0)^\top K^{-1}(M_{01}^{-1}-I)x_2 + (x_1-x_0)^\top K^{-1}(M_{12}^{-1} -M_{01}^{-1})(x_2-x_1)\\
& \qquad-x_0^\top K^{-1}(I-M_{12}^{-1})(x_2-x_1). 
\end{eqnarray*}
The forward difference of the second term in the integral is
\begin{eqnarray*}
\fl -\frac{1}{2}\Delta(x_0^\top K^{-1}(M_{01}^{-1}-I)x_1)\\
= -\frac{1}{2}\left(
x_0^\top K^{-1}(M_{01}^{-1}-I)\Delta x_1 + x_0^\top  K^{-1}\Delta(M_{01}^{-1}-I) x_2 + (\Delta x_0)^\top K^{-1}(M_{01}^{-1}-I)x_2\right)\\
= -\frac{1}{2}\left(x_0^\top K^{-1}(M_{01}^{-1}-I)(x_2-x_1)+x_0^\top K^{-1}(M_{12}^{-1}-M_{01}^{-1})x_2\right.\\
 \qquad \left.+ (x_1-x_0)^\top K^{-1}(M_{12}^{-1}-I)x_2\right)
\end{eqnarray*}
Each term is a quadratic form with coefficient matrix one of the symmetric matrices $K^{-1}$, $K^{-1}M_{01}^{-1}$, or $K^{-1}M_{12}^{-1}$.

The terms with coefficient matrix $K^{-1}M_{12}^{-1}$ are applied to the pairs of vectors
\begin{eqnarray*}
\fl (x_1,x_2)-(x_1,x_1)-(x_0,x_2)+(x_0,x_1)+(x_0,x_2)-(x_0,x_1)-\frac{1}{2}\left((x_0,x_2)+(x_1,x_2)-(x_0,x_2)\right) \\
 = -\frac{1}{2}(x_1,2x_1-x_2),
\end{eqnarray*}
giving the value
\begin{eqnarray}
\label{eq:term1}
-\frac{1}{2}x_1^\top K^{-1} M_{12}^{-1}(2x_1-x_2)=-\frac{1}{2}x_1^\top K^{-1} x_0.
\end{eqnarray}

The terms with coefficient matrix $K^{-1}$ are applied to the pairs of vectors
\begin{eqnarray*}
\fl - (x_1,x_2)+(x_0,x_2)-(x_0,x_2)+(x_0,x_1)+\frac{1}{2}\left((x_0,x_2)-(x_0,x_1)+(x_1,x_2)-(x_0,x_2)\right)\\
=\frac{1}{2}(-(x_1,x_2)+(x_1,x_0)),
\end{eqnarray*}
giving the value
\begin{eqnarray}
\label{eq:term2} -\frac{1}{2}x_1^\top K^{-1}(x_2-x_0).
\end{eqnarray}

The terms with coefficient matrix $K^{-1} M_{01}^{-1}$ are applied to the pairs of vectors
\begin{eqnarray*}
\fl  (x_1,x_2)-(x_0,x_2)-(x_1,x_2)+(x_1,x_1)+(x_0,x_2)-(x_0,x_1)-\frac{1}{2}((x_0,x_2)-(x_0,x_1)-(x_0,x_2))\\
 = \frac{1}{2}(x_1,2x_1-x_0)
\end{eqnarray*}
giving the value
\begin{eqnarray}
\label{eq:term3}
\frac{1}{2}x_1^\top K^{-1} M_{01}^{-1}(2x_1-x_0)= \frac{1}{2}x_1^\top K^{-1} x_2.
\end{eqnarray}

The forward difference $\Delta F(x_0,x_1)$ of the integral is the sum of the three terms (\ref{eq:term1}), (\ref{eq:term2}), and (\ref{eq:term3}), namely zero. 

\end{proof}

\noindent\textbf{Remark.} The denominator of the invariant measure (Prop. \ref{prop:measure}) is equal to the denominator of the first integral  (Theorem \ref{thm:integral}). Therefore, both the numerator and the denominator of the first integral are Darboux polynomials with cofactor equal to the reciprocal of the Jacobian determinant of the map. Exactly the same structure holds for the Kahan map Eq. \ref{eq:kahanmap} in the Hamiltonian case \cite{ce6}.

\section{Polarization of nonhomogeneous second order ODEs}

A nonhomogeneous function $f\colon\mathbb{R}^n\to \mathbb{R}$ of degree $\le d$ is associated with the homogeneous
function $\tilde f\colon \mathbb{R}^{n+1}\to\mathbb{R}$ of degree $d$ given by $\tilde f(x,z) := z^d f(x/z)$, where $z\in\mathbb{R}$. Therefore, nonhomogeneous potentials $W$ can be handled by polarization by (i) homogenising $\nabla W$;  (ii) extending $K$ to $\widetilde K = \left[\begin{smallmatrix}K & 0 \\  0 & 0\end{smallmatrix}\right]$; 
discretizing the extended homogeneous system by polarization; and (iv) restricting to $z=1$, noting that
the level set $z=1$ is
invariant under both the extended differential equation and the polarised map.

Proposition \ref{prop:rkn} below shows, however, that the method resulting from steps (i)--(iv) above and
the original method (\ref{eq:map}) for homogeneous potentials are both equivalent to the same
 Runge--Kutta--Nystr\"om method.

\def\pol{\mathrm{pol}}

For any (possibly nonhomogeneous) function or vector field $f(x)$ of degree $\le d$, let
\begin{eqnarray*} \pol_d f(x_0,\dots,x_{d-1})\end{eqnarray*}
be the restriction to $z=1$ of the polarization of the homogenization $\tilde f(x,z)$ of $f(x)$.
The polar map of $\ddot x = f(x)$, $\deg(f)\le 3$, is then defined by
\begin{eqnarray}
\label{eq:polar2}
 x_2-2x_1+x_0 = h^2 \pol_3 f(x_0,x_1,x_2).
 \end{eqnarray}

\begin{proposition} \cite{ho2}
\label{prop:rkn}
The map (\ref{eq:polar2}) is equivalent to a specific 7-stage  Runge--Kutta--Nystr\"om method.
\end{proposition}
 
\begin{proof}
We first establish an identity for arbitrary $d$ that provides a Runge--Kutta-like expression for the polarization of nonhomogeneous functions.

For homogeneous $f$ of degree $d$, a standard identity in algebraic polarization \cite[p. 110]{greenberg} recovers its
polarization as a linear combination of its values:
\begin{eqnarray*}
\pol_d f(x_0,\dots,x_{d-1}) = \frac{1}{d!} \sum_{1\le m \le d \atop 0\le i_1<\dots < i_m<d} (-1)^{k-m} f(x_{i_1}+\dots+x_{i_m}),
\end{eqnarray*}
where the sum is over all nonempty subsets of $\{0,\dots,d-1\}$.
Using homogeneity of $f$, we get
\begin{eqnarray}
\label{eq:polrk}
\pol_d f(x_0,\dots,x_{d-1})  = \frac{1}{d!} \sum_{1\le m \le d \atop 0\le i_1<\dots < i_m<d} (-1)^{k-m} m^{d} f\left(
\frac{x_{i_1}+\dots+x_{i_m}}{m}\right).
\end{eqnarray}
This is a linear combination of the values of $f$ at $2^{d}-1$ points, each of which is a convex combination of the $x_j$. 

For $d=3$ and the method (\ref{eq:polar2}),
these points may be taken to be the 7 stage values in $x$ of a  Runge--Kutta--Nystr\"om method.

If, now, $f$ is nonhomogeneous of degree $\le d$, then $\tilde f([x,1])=f(x)$ and (on the right hand side of (\ref{eq:polrk})  in the last argument of $\tilde f$) $(1+\dots+1)/m=1$, giving
\begin{eqnarray*}
\pol_d \tilde f([x_0,1],\dots,[x_{d-1},1])=\pol_d f(x_0,\dots,x_{d-1}),
\end{eqnarray*}
while the final zero row of $\widetilde K$ gives the discretization of the $z$ component of the extended system as
\begin{eqnarray*} z_2-2z_1+z_0=0\end{eqnarray*}
which is satisfied by $z_k=1$ for all $k$.
\end{proof}

\begin{table}
\small
\begin{eqnarray*}
\pol_2 f&= 2 f\left(\frac{x_0+x_1}{2}\right)-\frac{1}{2}f(x_0)-\frac{1}{2}f(x_1)\\
\pol_2 x^2 &= 2 \left(\frac{x_0+x_1}{2}\right)^2 - \frac{1}{2}x_0^2 - \frac{1}{2}x_1^2 = x_0 x_1\\
\pol_2 x &= 2\frac{x_0+x_1}{2} - \frac{1}{2}x_0 - \frac{1}{2}x_1 = \frac{1}{2}(x_0+x_1)\\
\pol_2 1 &= 2 - \frac{1}{2}-\frac{1}{2} = 1\\
\pol_3 f &= \frac{27}{6}f\left(\frac{x_0+x_1+x_2}{3}\right) 
- \frac{8}{6}f\left(\frac{x_0+x_1}{2}\right) 
- \frac{8}{6}f\left(\frac{x_0+x_2}{2}\right) 
- \frac{8}{6}f\left(\frac{x_1+x_2}{2}\right) \\ 
& \qquad+ \frac{1}{6}f(x_0)
+ \frac{1}{6}f(x_1)
+ \frac{1}{6}f(x_2)\\
\pol_3 x^3 &= x_0 x_1 x_2\\
\pol_3 x^2 &= \frac{1}{3}(x_0 x_1 + x_1 x_2 + x_2 x_0)\\
\pol_3 x &= \frac{1}{3}(x_0+x_1+x_2)\\
\pol_3 1 &= 1\\
\end{eqnarray*}
\caption{\label{tab:pol} Examples of degree-$d$ polarization of polynomials of degree $\le d$, for $d=2$ and $d=3$, illustrating how the Runge--Kutta-like formula Eq. \ref{eq:polrk}), applied to nonhomogenous polynomials, gives the same result as homogenization.}
\end{table}

The invariant measure for the nonhomogeneous case can be determined by applying Prop. \ref{prop:measure} to the homogenized system. It can also be expressed directly in terms of the nonhomogeneous vector field, as follows.

\begin{proposition}
\label{prop:fp}
Let $f=-K\nabla W$ be a nonhomogeneous  vector field of degree $\le 3$. The first-order map associated with
\ref{eq:polar2}) has invariant measure 
\begin{eqnarray}
\label{eq:measure2}
\frac{\prod dx}{\det\left(I-\frac{h^2}{3}\pol_2\D f\right)}.
\end{eqnarray}
\end{proposition}
\begin{proof}
First, if $f$ is homogeneous then 
from (\ref{eq:hess}) we have
\begin{eqnarray*}
-K V(x_0,x_1,x_2,\cdot) &= -K \pol_3 \nabla W(x_0,x_1,x_2) \\
&= \pol_3 f(x_0,x_1,x_2)
\end{eqnarray*}
and
\begin{eqnarray*}
-K V(x_0,x_1,\cdot,\cdot) 
&= -K\frac{1}{3}\pol_2\mathrm{Hess} W(x_0,x_1)\\
& = \frac{1}{3} \pol_2\D f(x_0,x_1),\\
\end{eqnarray*}
establishing (\ref{eq:measure2}) for the homogeneous case.

If $f$ is nonhomogeneous, we apply (\ref{eq:measure2}) to the homogenized system and restrict the invariant measure
to the invariant set $z=1$. The last row and column of $\widetilde K$ are zero, and the top left $n\times n$ block
of $\D\tilde f$, evaluated at $z=1$, is equal to $\D f$. This gives the result.
\end{proof}

The integral of the polar map for nonhomogeneous $f$ cannot be obtained directly, as the extended matrix $\widetilde K$ is noninvertible. Instead, we introduce a perturbation parameter $\varepsilon$ and let
$\widetilde K = \left[\begin{smallmatrix}K & 0 \\  0 & \varepsilon\end{smallmatrix}\right]$. Then
the polar map of the perturbed homogenized system has the form
\begin{eqnarray*}
x_2-2x_1+x_0 &= -h^2 K \nabla_x\widetilde V(x,z)\\
z_2-2z_1+z_0 &= -h^2\varepsilon \nabla_z\widetilde V(x,z).
\end{eqnarray*}
Its first integral  from (\ref{eq:int}) has the form
\begin{eqnarray*}
F([x_0,z_0],[x_1,z_1])=\frac{1}{\varepsilon} (z_1-z_0)^2 + F_0([x_0,z_0],[x_1,z_1]) + \mathcal{O}(\varepsilon).
\end{eqnarray*}
Restricting to $z_0=z_1=1$, i.e., passing to the polar map (\ref{eq:polar2}) of the nonhomogeneous system, and then taking the limit $\varepsilon\to 0$, determines $F_0([x_0,1],[x_1,1])$ as a first integral of (\ref{eq:polar2}). 

\section{Higher order systems}
The following proposition generalises Prop. \ref{prop:fp} to $m$th order systems featuring degree $m+1$ polynomials.
The $m=2$ case generalizes Prop. \ref{prop:fp} to arbitrary nonconservative vector fields $f$.  The odd $m$ case contains a restriction on $f$ which is satisfied, for example, by Hamiltonian vector fields when $m=1$. \begin{proposition}
\label{prop:ho}
The polar map
\begin{eqnarray}
\label{eq:ho}
x_m +\left(\sum_{i=1}^{m-1}c_i x_i\right)+ (-1)^m x_0 = h^m \pol_{m+1} f(x_0,\dots,x_m)
\end{eqnarray}
associated with the $m$th-order polynomial differential equation
\begin{eqnarray*}
x^{(m)} = f(x),\quad x\in\mathbb{R}^n,\ \mathrm{deg}(f)\le m+1
\end{eqnarray*}
preserves the measure
\begin{eqnarray*}\frac{\prod dx}{\det(I-\frac{h^m}{m+1} \pol_m\D f)}\end{eqnarray*}
\begin{itemize}
\item[(i)] for all $f$ when $m$ is even; and
\item [(ii)] for all $f$ that satisfy $\det(I+\D f)=\det(I-\D f)$ when $m$ is odd.
\end{itemize}
\end{proposition}
\begin{proof}
First consider the case that $f$ is homogeneous, in which case the map (\ref{eq:ho}) can be written in the form
\begin{eqnarray*}
x_m + \left(\sum_{i=1}^{m-1}c_i x_i\right)+ (-1)^m x_0 = \frac{1}{m+1}h^m \pol_m \D f(x_0,\dots,x_{m-1})x_m
\end{eqnarray*}
or
\begin{eqnarray}
\label{eq:Mform2}
(I-h^m \frac{1}{m+1} \pol_m\D f(x_0,\dots,x_{m-1}))x_m = -\left(\sum_{i=1}^{m-1}c_i x_i\right)- (-1)^m x_0.
\end{eqnarray}

The map in first order form
\begin{eqnarray*} \begin{pmatrix}x_{m-1} \\  \vdots \\ x_0 \end{pmatrix}
\mapsto
\begin{pmatrix}x_m\\  \vdots\\x_1 \end{pmatrix}
\end{eqnarray*}
has Jacobian derivative
\begin{eqnarray*}\begin{pmatrix} \D_{x_{m-1}} x_m & \D_{x_{m-2}} x_m & \dots & \D_{x_0} x_m\\ 
 I & 0 & \dots & 0 \\ 
 0 & I & \dots & 0 \\ 
\dots & & & \\ 
 \end{pmatrix}
\end{eqnarray*}
with determinant
\begin{eqnarray}
\label{eq:detm}
\det((-1)^{m+1}\D_{x_0} x_m).
\end{eqnarray}

Taking Jacobian derivatives of both sides of (\ref{eq:Mform2}) with respect to $x_0$ and using the product rule on the left gives
\begin{eqnarray*}
 (I-h^m\frac{1}{m+1} \pol_m\D f(x_0,\dots,x_{m-1}))&  \D_{x_0} x_m- h^m\frac{1}{m+1} \pol_m\D f(x_1,\dots,x_m))
 \\ & = (-1)^{m+1}I
 \end{eqnarray*}
or
\begin{eqnarray*}
 (I-h^m\frac{1}{m+1} \pol_m & \D f(x_0,\dots,x_{m-1}))  \left((-1)^{m+1}\D_{x_0} x_m\right)\\
&= I+(-1)^{m+1} h^m\frac{1}{m+1} \pol_m\D f(x_1,\dots,x_m)
\end{eqnarray*}
Taking determinants gives
\begin{eqnarray*}
 \det\big(I-h^m\frac{1}{m+1} \pol_m & \D f(x_0,\dots,x_{m-1})\big)  \det\left((-1)^{m+1})\D_{x_0} x_m\right)\\
& = \det\big(I+(-1)^{m+1} h^m\frac{1}{m+1} \pol_m\D f(x_1,\dots,x_m)\big)
\end{eqnarray*}
Together with Eq. (\ref{eq:detm}), and $\det(I+\D f)=\det(I-\D f)$ when $m$ is odd, this establishes the result.
The nonhomogeneous case follows as in Prop. \ref{prop:fp}.
\end{proof}
Note that if $f=J^\top \nabla W$ for some $W\colon\R^n\to \R$ and some antisymmetric matrix $J$, then Sylvester's criterion gives
 $\det(I+\D f) = \det(I+J \mathrm{Hess} W) = \det(I+(\mathrm{Hess} W)J)
= \det(I+J^\top (\mathrm{Hess} W)^\top) = \det(I-\D f)$. However, other vector fields (such as
$v(x)\frac{\partial}{\partial y} + w(y)\frac{\partial}{\partial x}$) also satisfy this condition.

\section{A system with a linear symmetry}

As we have seen, the discretization considered in this paper is equivalent to a Runge--Kutta--Nystr\"om method. Therefore, it is equivariant with respect to (partitioned) linear maps \cite{mo}.  That is, given any linear maps $A\colon X\to Y$, and any two $A$-related differential equations $\ddot x = f(x)$ ($x\in X$) and $\ddot y = g(y)$ ($y\in Y$) obeying
$g\circ A = A\circ f$, the polar maps associated with the differential equations are themselves $A$-related. If $A$ is invertible and $g=f$, then $A$ is a symmetry of $f$ and of its polar map. (Similar remarks hold for affine maps). 

This suggests that polar maps of differential equations with linear symmetries may have further special properties.
In this section we consider the special case of the  system with $n=2$ and rotationally-invariant Hamiltonian
\begin{eqnarray}
\label{eq:rot}
H(x,p) = \frac{1}{2}\|p\|^2 + \alpha\|x\|^2+\beta\|x\|^4.
\end{eqnarray}
The symmetry $x\mapsto Rx$, $R\in O(2)$, of Hamilton's equations extends to a symmetry $(x_0,x_1)\mapsto (Rx_0,Rx_1)$ of the polar map. The invariant measure of this map is rotationally invariant. Computing this measure using Proposition \ref{prop:measure} gives its explicit expression
$$ \left( (1 + \textstyle\frac{2}{3}h^2(\alpha + 4\beta(x_0\cdot x_1)))^2 - \frac{16}{9}h^4\beta^2 \|x_0\|^2\|x_1\|^2\right)^{-1}dx_{0,1}\wedge dx_{0,2}
\wedge dx_{1,1}\wedge dx_{1,2}$$
in terms of the three fundamental invariants $x_0\cdot x_1$, $\|x_0\|^2$, and $\|x_1\|^2$ of the group action. (Here $x_0=(x_{0,1},x_{0,2})\in\mathbb{R}^2.$)

The Jacobian determinant of the polar map factors as $N_2 N_6/D_4^3$, where the subscripts denote the degree of the polynomials. This provides more candidates for cofactors.

Indeed, we find immediately that the cofactor $N_2/D_4$ has two rotationally-invariant Darboux polynomials $P_1$ and $P_2$:
\begin{eqnarray*}
P_1 &= x_{1,1} x_{0,2} - x_{0,1} x_{1,2}\\  
P_2 &= 3 + h^2(2\alpha + 4\beta(x_0\cdot x_1)).
\end{eqnarray*}
These provide the integral $P_1/P_2$, that approximates $\frac{1}{3}h$ times the integral of Hamilton's equations that is associated with the rotational symmetry, namely $x_2\dot x_1 - x_1 \dot x_2$. 

Trying other cofactors and Darboux polynomials up to degree 6 yields no further information other than those known for general systems and those generated by them.

The symmetry, integrals, and invariant measure are enough to ensure that the map is superintegrable in the sense of van der Kamp et al. \cite{vanderkamp}, i.e, it is a measure-preserving $n$-dimensional map with $n-1$ functionally independent constants of motion.

\begin{proposition}The polar map associated with (\ref{eq:rot}) is superintegrable.
\end{proposition}
\begin{proof}
The map is 4-dimensional with a 1-dimensional measure-preserving symmetry group, and thus descends to a measure-preserving map on the 3-dimensional quotient \cite{huang}. The two integrals are also invariant under the symmetry and hence descend to the quotient.  This yields a 3-dimensional measure-preserving map with 2 functionally independent integrals, thus superintegrable.
\end{proof}

A different integrable discretization of this system is given in McLachlan \cite{mc}.

\paragraph{Acknowledgements}
GRWQ is grateful to SMRI for support during a visit to the University of Sydney, and to Nalini Joshi and colleagues for fruitful and enjoyable discussions.

 \end{document}